\documentclass{article}

\usepackage[margin=1in]{geometry}
\usepackage[all]{xy}
\usepackage{amssymb}
\usepackage{amsfonts}
\usepackage{amsthm}
\usepackage{amsmath}
\usepackage{hyperref}
\usepackage{mathtools}
\usepackage{url}
\usepackage{color}

\newcommand{\Addresses}{{
		\bigskip
		\footnotesize
		
		\textsc{Department of Mathematics, Technion - Israel Institute of Technology, Haifa, Israel}\par\nopagebreak
		\textit{E-mail address:} \texttt{ofir.gor@technion.ac.il}
}}

\theoremstyle{plain}
\newtheorem*{thm*}{Theorem}
\newtheorem{thm}{Theorem}[section]

\newtheorem{lem}[thm]{Lemma}  
\newtheorem{proposition}[thm]{Proposition}
\newtheorem{cor}[thm]{Corollary}

\theoremstyle{remark}
\newtheorem{remark}{Remark}

\newcommand{\CC}{\mathbb{C}}
\newcommand{\PP}{\mathbb{P}}

\newcommand{\FF}{\mathbb{F}}
\newcommand{\RR}{\mathbb{R}}
\newcommand{\xii}{\widetilde{\xi}}
\newcommand{\uu}{\widetilde{u}}

\numberwithin{equation}{section}

\begin{document}
	\title{Smooth permutations and polynomials revisited}
	\author{Ofir Gorodetsky}
	\date{}
	\maketitle
	
\begin{abstract}
We study the counts of smooth permutations and smooth polynomials over finite fields. For both counts we prove an estimate with an error term that matches the error term found in the integer setting by de Bruijn more than 70 years ago. The main term is the usual Dickman $\rho$ function, but with its argument shifted. 

We determine the order of magnitude of $\log(p_{n,m}/\rho(n/m))$ where $p_{n,m}$ is the probability that a permutation on $n$ elements, chosen uniformly at random, is $m$-smooth.

We uncover a phase transition in the polynomial setting: the probability that a polynomial of degree $n$ in $\FF_q$ is $m$-smooth changes its behavior at $m\approx (3/2)\log_q n$. \end{abstract}
\section{Introduction}
A permutation is said to be $m$-smooth if it has no cycles of length greater than $m$. Let 
\[ p_{n,m} :=\PP_{\pi \in S_n} (\pi \text{ is }m\text{-smooth})\]
be the probability that a permutation from $S_n$ chosen uniformly at random is $m$-smooth. Let
\[ u := \frac{n}{m}\]
and let $\rho\colon [0,\infty) \to (0,\infty)$ be the Dickman function, defined via the delay differential equation $t\rho'(t)+\rho(t-1) =0$ for $t>1$ with initial conditions $\rho(t)=1$ for $t \le 1$. It is weakly decreasing, and de Bruijn proved that it satisfies
\begin{equation}\label{eq:db} \rho(u) = \exp\left( -u\log (u \log u)+ u + O\left(\frac{u \log \log u}{\log u}\right)\right)
\end{equation}
for $u \ge 3$  \cite{debruijn1951}. Goncharov \cite{goncharov1944} established a connection between $p_{n,m}$ and $\rho$, showing that $p_{n,m} \sim \rho(u)$ as $n \to \infty$ and $u = O(1)$.
In an impressive work, Manstavi\v{c}ius and Petuchovas \cite{Manstavicius2016} established asymptotic estimates for $p_{n,m}$ in the entire range $1 \le m \le n$, including the estimate \cite[Thm.~4]{Manstavicius2016}
\begin{equation}\label{eq:pnmboundweaker}
	p_{n,m} = \rho(u) \left(1+ O\left( \frac{u \log (u+1)}{m}\right)\right)
\end{equation}
which holds for $n \ge m \ge \sqrt{n \log n}$. Recently, Ford \cite[Thm.~1.17]{Ford2021} proved that 
\begin{equation}\label{eq:ford}
	\rho\left( \frac{n}{m}\right) \le p_{n,m} \le \rho\left(\frac{n+1}{m+1}\right)
\end{equation}
holds uniformly for $n \ge m \ge 1$. This almost immediately leads to
\begin{equation}\label{eq:pnmbound}
	p_{n,m} = \rho(u) \exp\left(O\left( \frac{u \log (u+1)}{m}\right)\right)
\end{equation}
holding in $n \ge m \ge 1$ \cite[Prop.~1.8]{Gorodetsky2022}, extending \eqref{eq:pnmboundweaker}. Theorems \ref{thm:between}--\ref{thm:mediumm} below, whose proofs borrow heavily from \cite{Manstavicius2016}, improve on \eqref{eq:pnmbound}.
\begin{thm}\label{thm:between}
Uniformly for $n \ge m \ge \sqrt{n \log n}$ we have
\begin{equation}\label{eq:rholarge2}
	p_{n,m} = \rho\left(\frac{n}{m+\frac{1}{2}}\right)\left(1+O\left(\frac{\log(u+1)}{m}\right)\right).
\end{equation}
	Uniformly for $\sqrt{n\log n} \ge m \ge 1$ we have
\[
		p_{n,m} =  \rho\left(\frac{n}{m+\frac{1}{2}}\right) \exp\left(O\left(\frac{u\log^2(u+1)}{m^2} \right)\right).\]
\end{thm}
We may combine both parts of Theorem~\ref{thm:between} as
\[ 
p_{n,m} =  \rho\left(\frac{n}{m+\frac{1}{2}}\right) \exp\left(O\left(\frac{\log(u+1)}{m}+\frac{u\log^2(u+1)}{m^2}  \right)\right)\]
uniformly for $n \ge m \ge 1$. We see $p_{n,m} \sim \rho\left(\frac{n}{m+\frac{1}{2}}\right)$ holds
when  $m/(n^{1/3}(\log n)^{2/3})\to \infty$.

The error term $O(\log(u+1)/m)$ in \eqref{eq:rholarge2} is of particular importance, as we now explain. We say that a positive integer is $y$-smooth if all its prime factors are at most $y$, and we denote \[\Psi(x,y) := \#\{ 1\le n \le x: n \text{ is }y\text{-smooth}\}.\] Setting $n':=\log x$, $m':=\log y$ and $u':=n'/m'$, de Bruijn \cite{debruijn19512} showed that, in the range $n' \ge m' \ge n'^{5/8+\varepsilon}$, the estimate \[ \frac{\Psi(x,y)}{x} = \rho(u')\left(1+O_{\varepsilon}\left(\frac{\log (u'+1)}{m'}\right)\right)\] 
holds. This error term, $O_{\varepsilon}\left(\frac{\log (u'+1)}{m'}\right)$, is analogous to \eqref{eq:rholarge2}.
\begin{remark}
The main term $\rho\left(\frac{n}{m+\frac{1}{2}}\right)$ of Theorem~\ref{thm:between} arises indirectly. We first prove an estimate with a more complicated main term, see Proposition~\ref{prop:between}. This term originally appeared in \cite[Cors.~3, 5]{Manstavicius2016}, where it serves as an approximation to $p_{n,m}$ in a narrow range and with a worse error term compared to Theorem~\ref{thm:between}. In Lemma~\ref{lem:shif} we simplify the complicated main term and show it may be replaced by $\rho\left(\frac{n}{m+\frac{1}{2}}\right)$ at a negligible cost.
\end{remark}
\begin{remark}
Ford's proof of \eqref{eq:ford} avoids complex analysis and it will be interesting to have a similar argument estimating $p_{n,m}$ relative to $\rho\left(\frac{n}{m+\frac{1}{2}}\right)$.
\end{remark}
Theorem \ref{thm:between} is weaker than \eqref{eq:pnmbound} when $m=o(\log n)$. We complement it with
\begin{thm}\label{thm:mediumm}
If $2\le m\le n$ satisfies $m=O(\log n)$ then
\[		p_{n,m} =\left( \frac{e}{n}\right)^u \exp\bigg( u \bigg(-\log\big(1-n^{-\frac{1}{m}}\big)+O\bigg( \frac{n^{-\frac{1}{m}}}{m} \bigg) \bigg)\bigg).\]
\end{thm}
The proof of Theorem \ref{thm:mediumm} relies on estimates from \cite{Manstavicius2016}. Theorem~\ref{thm:mediumm} is only claimed for $m \ge 2$, as it gives a wrong estimate if applied with $m=1$. By Stirling's approximation, $p_{n,1}=1/n! \asymp (e/n)^n n^{-1/2}$ if $m=1$. 

From \eqref{eq:pnmbound} and Theorem~\ref{thm:mediumm} we see that $\log p_{n,m} \sim \log \rho(u)$ holds as $n\to \infty$.  Using Theorems \ref{thm:between}--\ref{thm:mediumm} we determine the order of magnitude of  $\log(p_{n,m}/\rho(u))$:
\begin{cor}\label{cor:lower}
	Define $A$ via $p_{n,m}=\rho(u)\exp(A)$. Uniformly for $n \ge m \ge 1$,
	\[ A \asymp  u \log \left(1 + \frac{\log u}{m}\right).\]
\end{cor}
\subsection{Results for smooth polynomials}\label{sec:polybg}
A polynomial over a finite field $\FF_q$ is said to be $m$-smooth if all its irreducible factors have degrees at most $m$. Let $\mathcal{M}_{n,q}$ be the set of monic polynomials of degree $n$ over $\FF_q$. We write $p_{n,m,q}$ for the probability that a polynomial from $\mathcal{M}_{n,q}$ chosen uniformly at random is $m$-smooth. Manstavi\v{c}ius \cite[Thm.~2]{Manstavicius19922} proved that
\[p_{n,m,q} = \rho(u) \left( 1 +O\left(\frac{u\log(u+1)}{m}\right)\right)\]
holds in range $n\ge m \ge \sqrt{n\log n}$.
Recently, the author proved that for fixed $\varepsilon>0$, the ratio $p_{n,m,q}/p_{n,m}$ satisfies 
\begin{equation}\label{eq:combined}
	\frac{p_{n,m,q}}{p_{n,m}} =1+ O_{\varepsilon}\left(\frac{u n^{\frac{1+\mathbf{1}_{2 \mid m}}{m}} \min\{ m ,\log(u+1)\}}{m q^{\lceil \frac{m+1}{2}\rceil}}\right)
\end{equation}
for $n \ge m \ge (2+\varepsilon)\log_q n$ \cite[Thms.~1, 2]{Gorodetsky2022}. The implied constant does not depend on $q$. From \eqref{eq:combined} and Theorem~\ref{thm:between} we immediately obtain
\begin{cor}
Uniformly for $n\ge m\ge \sqrt{n\log n}$ we have  
\[p_{n,m,q}=\rho\left( \frac{n}{m+\frac{1}{2}} \right) \left(1+O\left( \frac{\log(u+1)}{m}\right)\right).\]
Fix $\varepsilon>0$. Uniformly for $\sqrt{n\log n}\ge m \ge (2+\varepsilon)\log_q n $ we have \[p_{n,m,q} = \rho\left(\frac{n}{m+\frac{1}{2}}\right) \exp\left( O_{\varepsilon}\left( \frac{u\log^2(u+1)}{m^2} \right)\right).\]
\end{cor}
The rest of our results concern  $p_{n,m,q}/p_{n,m}$. To obtain estimates with better range and error than \eqref{eq:combined} we introduce a function $G_q(z)$. The generating function of $\{p_{n,m,q}\}_{n \ge 0}$ is
\begin{equation}\label{eq:defFq}
	F_q(z):=\prod_{\substack{P \in \mathcal{P} \\ \deg(P) \le m}} \bigg(1- \bigg(\frac{z}{q}\bigg)^{\deg (P)}\bigg)^{-1}
\end{equation}
where $\mathcal{P}=\mathcal{P}_{q}$ is the set of monic irreducible polynomials over $\FF_q$. It is analytic in $|z|<q$. The generating function of $\{p_{n,m}\}_{n \ge 0}$ is the entire function \cite{Manstavicius2016}
\begin{equation}\label{eq:Fdef} F(z):=\exp\left(\sum_{i=1}^{m} \frac{z^i}{i}\right).
\end{equation}
For $|z|<q$ we define
\begin{equation}\label{eq:Gqdeforig}
G_q(z) := \frac{F_q(z)}{F(z)}.
\end{equation}
The function $G_q$ is studied in \cite[Lem.~2.1]{Gorodetsky2022}. By \cite[Thm.~1.3]{Gorodetsky2022} we have 
\begin{equation}\label{eq:main mediumn32}
	\frac{p_{n,m,q}}{p_{n,m}} =  G_q(x) \bigg(1+ O_{\varepsilon}\bigg(\frac{n^{\frac{1+\mathbf{1}_{2 \mid m}}{m}} \min\{m, \log (u+1)\}}{mq^{\lceil \frac{m+1}{2}\rceil}}  \bigg)\bigg)
\end{equation}
in the range $n/(\log n \log^3 \log(n+1)) \ge m \ge (2+\varepsilon)\log_q n$, where $x=x_{n,m}$ is the unique positive solution to
\begin{equation}\label{eq:saddledef2}
	\sum_{i=1}^{m} x^i = n.
\end{equation}
In that range, $G_q(x)$ is very close to $1$, see  \cite[Lems.~5.3-5.4]{Gorodetsky2022}.
We complement \eqref{eq:main mediumn32} with a result for large $m$. We define $\xi(u)\ge 0$ by $e^{\xi(u)} = 1+u\xi(u)$.
\begin{thm}\label{thm:smallupnmq}
If $n \ge m\ge \sqrt{10n \log n}$ then
\[ \frac{p_{n,m,q}}{p_{n,m}} = G_q(e^{\xi(u)/m}) \left(1 + O \left( \frac{\log(u+1)}{m q^{\lceil \frac{m+1}{2}\rceil}}\right)\right).\]
\end{thm}
Using Theorem~\ref{thm:smallupnmq} we are able to drop the condition $n/(\log n\log^3 \log(n+1)) \ge m$ present in \eqref{eq:main mediumn32}:
\begin{cor}\label{cor:rangelarger}
The estimate \eqref{eq:main mediumn32} holds uniformly for $n \ge m \ge (2+\varepsilon)\log_q n$.
\end{cor}
We turn to smaller $m$. In the range $(2-\varepsilon)\log_q n \ge m \ge (1+\varepsilon)\log_q n$,
\[ \log G_q(x) \asymp_{q,\varepsilon} \frac{n^2}{mq^m},\]
see \cite[Lem.~5.4]{Gorodetsky2022}. In \cite[Thm.~1.3]{Gorodetsky2022} it is shown that $p_{n,m,q}/p_{n,m} \sim G_q(x)$ holds as $n \to \infty$ as long as $m \ge (3/2+\varepsilon)\log_q n$. Our next result shows that $p_{n,m,q}/p_{n,m}$ experiences a transition when $m$ is close to $(3/2)\log_q n$.
\begin{thm}\label{thm:trans}
Fix a prime power $q$ and $\varepsilon>0$. There is a quantity $A>0$ of order
\[ A \asymp_{q,\varepsilon} \frac{n^3}{mq^{2m}}\]
such that, for $(2-\varepsilon)\log_q n \ge m \ge (1+\varepsilon)\log_q n$ we have, as $n \to \infty$,
\[\frac{p_{n,m,q}}{p_{n,m}} \sim G_q(x)\exp( -A).\]
\end{thm}
In particular, if $q^m=n^{\tau}$ where $\tau\in(1,2)$ then $A\asymp_{q,\varepsilon} \frac{n^{3-2\tau}}{\log n}$ and so $A\gg_{q,\varepsilon}1$ if $q^m =O( n^{3/2}/(\log n)^{1/2})$.
\subsection*{Conventions}
The letters $C,c$ shall denote absolute positive constants which may change between different instances. The notation $A \ll B$ means $|A| \le C B$ for some absolute constant $C$, while $A\ll_{a,b,\ldots} B$ means $C$ may depend on the parameters in the subscript. We write $A\asymp_{a,b,\ldots}B$ to indicate $cB \le A \le C B$ holds for $C,c>0$ that may depend on $a,b,\ldots$. 
\section{Auxiliary estimates and functions}
\subsection{Saddle point review}
By Theorem 2 and Corollary 5 of \cite{Manstavicius2016}, uniformly for $n \ge m \ge 1$ we have
\begin{equation}\label{eq:saddle}
	p_{n,m}= \frac{\exp\left( \sum_{i=1}^{m} \frac{x^i}{i}\right)}{x^n \sqrt{2\pi \lambda}} \left(1+O\left( u^{-1}\right)\right)
\end{equation}
where $x=x_{n,m}$ is the saddle point defined as the unique positive solution to \eqref{eq:saddledef2}, and $\lambda=\lambda_{n,m}$ is defined as
\begin{equation}\label{eq:lambedadef}
	\lambda = \sum_{i=1}^{m} ix^i.
\end{equation}
\begin{lem}\label{lem:xlam}
Uniformly for $n \ge m \ge 1$, the following estimates hold: $n^{1/m} \ge x \gg n^{1/m}$, $nm \ge \lambda \gg nm$ and $\lambda= nm( 1+ O(1/\max\{ \log u,n^{1/m}\}))$. 
\end{lem}
\begin{proof}
We study $x$. Considering only the $i=m$ term in \eqref{eq:saddledef2} gives $x \le n^{1/m}$. If $x<1$ then $\sum_{i=1}^m x^i < m \le n$, a contradiction, hence $x \ge 1$. Consequently, by the definition of $x$, $mx^m \ge n$ holds, which implies
\[	x \ge  n^{\frac{1}{m}} m^{-\frac{1}{m}} \ge
 n^{\frac{1}{m}}e^{-\frac{1}{e}}.\]
	We turn to $\lambda$. We have $\lambda  \le m \sum_{i=1}^{m} x^i = nm$. By \cite[Lem.~9]{Manstavicius2016}, \begin{equation}\label{eq:lem9}
		\lambda = nm(1+O(1/\log u))
	\end{equation}
	uniformly for $u >1$. We also have $x^m \le n=\sum_{i=1}^{m} x^i < x^m/(1-x^{-1})$ and so
\[	mn(1-x^{-1}) < mx^m \le \sum_{i=1}^{m}ix^i = \lambda \le nm \]
	implying $\lambda=nm(1+O(x^{-1}))=nm(1+O(n^{-1/m}))$. The estimate \eqref{eq:lem9} already shows $\lambda  \gg nm$ when $u \ge C$. If $u < C$ then, since $x \ge 1$, we find $\lambda \ge \sum_{i=1}^{m} i \gg nm$.
\end{proof}
\subsection{Dickman function review}
We define $\xi \colon [1,\infty) \to [0,\infty)$, a function of variable $u\ge 1$, by	\[e^{\xi(u)} = 1+u\xi(u).\]
\begin{lem}\cite[Lem.~1]{hildebrand1984}\label{lem:xi size}
	We have $\xi \sim \log u$ as $u \to \infty$, and $\xi'=u^{-1}(1+O(1/\log u))$ for $u \ge 2$.
\end{lem}
\begin{lem}\cite[Lem.~6]{Manstavicius2016}\label{lem:uniformxi}
If $u>1$ we have $\log u < \xi \le 2 \log u$.
\end{lem}
Let
\[	I(s) = \int_{0}^{s} \frac{e^t-1}{t} {\mathrm d}t\]
which defines an entire function. Observe that 
\[ I'(\xi) = \frac{e^{\xi}-1}{\xi}=u,\]
and a short computation shows $I''(\xi)=1/\xi'$. The following lemma is proved by induction.
\begin{lem}\label{lem:Iform}
	For any $k \ge 1$ we have $I^{(k)}(s) = (e^s P_{k-1}(s)-(-1)^{k-1} (k-1)!)/s^k$ for a monic polynomial $P_{k-1}$ of degree $k-1$. 
\end{lem}
It has the following direct consequence.
\begin{cor}\label{cor:isize}
	For any fixed $k \ge 1$ we have $I^{(k)}(\xi) =(1+O_k(1/\log u))\frac{e^{\xi}}{\xi}\sim u$ as $u \to \infty$, and $I^{(k)}(\xi+it) \ll_k \min\{u, e^{\xi(u)}/|\xi+it|\}$ uniformly for $t \in \RR$.
\end{cor}
The function $I$ arises when studying $\rho$ and its Laplace transform.
\begin{lem}\cite[Eq.~(1.9)]{Debruijn}\cite[Eq.~(3.9)]{Alladi}\label{lem:rho i transform}
Let $\gamma$ be the Euler--Mascheroni constant. We have
\[	\hat{\rho}(s) := \int_{0}^{\infty} e^{-sv}\rho(v)\, {\mathrm d}v = \exp\left( \gamma + I(-s) \right)\]
for $s \in \CC$. Uniformly for $u \ge 1$,
	\begin{equation}\label{eq:rho and i}
		\rho(u)=\frac{1}{\sqrt{2\pi I''(\xi)}}\exp\left( \gamma-u\xi + I(\xi)\right)(1 + O( u^{-1})).
	\end{equation}
\end{lem}
We have the following bounds on $\hat{\rho}$.
\begin{lem}\cite[Lem.~2.7]{HT}\label{lem:i bounds}
	The following bounds hold for $s=-\xi(u)+it$:
\[		\hat{\rho}(s) = \begin{cases} O\left(\exp\left(I(\xi)-\frac{t^2u}{2\pi^2}\right)\right) & \mbox{if }|t| \le \pi,\\
			O\left(\exp\left(I(\xi)-\frac{u}{\pi^2+\xi^2}\right)\right) & \mbox{if }|t| \ge \pi,\\
			\frac{1}{s} + O\left( \frac{1+u \xi }{|s|^2} \right) & \mbox{if }1+u\xi =O(|t|).\end{cases}\]
\end{lem}
\subsection{\texorpdfstring{$T$}{T}}
We define the following function, holomorphic in the strip $|\Im s |<2\pi m$:
\[	T(s) := \int_{0}^{s} \frac{e^t-1}{t}\left( \frac{\frac{t}{m}}{1-e^{-t/m}}-1\right){\mathrm d}t.\]
The Bernoulli numbers $B_i$ are defined by
\[\frac{s}{1-e^{-s}} = 1+\frac{s}{2} +\frac{s^2}{12}-\frac{s^4}{720}+ \ldots = \sum_{i\ge 0} B_i \frac{s^i}{i!}.\]
This series converges for $|\Im s|<2\pi$. It is known that $B_i \ll i!/(2\pi)^i$. It is then easy to see, from the estimate $\int_{0}^{s} (e^t-1)t^{i-1} {\mathrm d}t \ll  |s|^{i-1} (|e^s|+|s|)$ which holds for $\Re s \ge 0$ and $i \ge 1$, that
	\begin{equation}\label{eq:estT}
	T(s) = \sum_{i=1}^{k}  \frac{B_i}{m^i i!} \int_{0}^{s} (e^t-1)t^{i-1} {\mathrm d}t + O\bigg( \frac{|s|^k (|e^{s}|+|s|)}{(2\pi m)^{k+1}} \bigg)
\end{equation}
	holds for $k \ge 0$ and $s=\sigma+it$ with $\sigma \in [0,\pi m]$ and $t \in [-\pi m,\pi m]$. Applying \eqref{eq:estT} with $k=1$ we obtain
\begin{cor}\cite[Lem.~11]{Manstavicius2016} \label{cor:t}
	For $s=\sigma+it$ with $\sigma\in [0,\pi m]$ and $t\in [-\pi m,\pi m]$,
\[	\left|T(s)+\frac{s}{2m}\right| \ll \frac{e^{\sigma}}{m} + \frac{t^2}{m^2}.\]
\end{cor}
Applying \eqref{eq:estT} with $k=2$ and $s=\xi=\xi(u)$ we obtain
\begin{equation}\label{eq:lemT1}
T(\xi) = \frac{(u-1)\xi}{2m} + \frac{e^{\xi}(\xi-1)-\frac{\xi^2}{2}+1}{12m^2} + O\bigg( \frac{u\log^3 (u+1)}{m^3}\bigg)
\end{equation}
when $\xi(u) \le \pi m$ holds.
\begin{lem}\label{lem:Tderivs}
	Suppose $\xi=\xi(u) \le \pi m$. For any $k \ge 1$ we have
	\begin{equation}\label{eq:tdev1}
	\begin{split}
		T^{(k)}(\xi) &= I^{(k)}(\xi)\left( \frac{\xi/m}{1-e^{-\xi/m}}-1+O_k\left( \frac{1}{m}\right)\right) \\
		&=\frac{e^{\xi}}{2m}\left(1+O_k\left(\frac{1}{\log(u+1)}+\frac{\log(u+1)}{m}\right)\right).
	\end{split}
	\end{equation}
If additionally $t \in [-\pi m, \pi m]$ then $T^{(k)}(\xi+it) \ll_k e^{\xi(u)}/m$.
\end{lem}
\begin{proof}
	We have $T'(s) = I'(s) g(s/m)$ for $g(z):=z/(1-e^{-z}) - 1$. Repeated differentiation shows
	\begin{equation}\label{eq:Tk1}
	T^{(k)}(s) = I^{(k)}(s) g(s/m)+O_k\bigg( \sum_{i=1}^{k-1}m^{-i}  |I^{(k-i)}(s)g^{(i)}(s/m)|\bigg).
\end{equation}
	We suppose that $s=\xi(u) + it$ and that the inequalities $\xi(u)\le \pi m$ and $|t|\le \pi m$ are satisfied. From \eqref{eq:Tk1},
	\begin{equation}\label{eq:Tk2} T^{(k)}(s) =I^{(k)}(s) g(s/m)+O_k\bigg( \frac{1}{m}\sum_{i=1}^{k-1}  |I^{(k-i)}(s)|\bigg)
		\end{equation}
	as our assumptions on $s$ guarantee that $g^{(i)}(s/m)\ll_i 1$. We first suppose $t=0$ and establish \eqref{eq:tdev1}. Using \eqref{eq:Tk2} and the definition of $g$, the first equality in \eqref{eq:tdev1} will follow if we show that $I^{(k-i)}(\xi(u)) \ll_k I^{(k)}(\xi(u))$, which is a consequence of Corollary~\ref{cor:isize}. The second equality in \eqref{eq:tdev1} follows from another application of Corollary~\ref{cor:isize}. Finally, the claim for $T^{(k)}(s)$ (when $|t|\le \pi m$) follows from \eqref{eq:Tk2} since $|I^{(k)}(s)| \ll_k e^{\xi(u)}/|s|$ and $|I^{(k-i)}(s)| \ll_k u$ ($1\le i \le k-1$) by Corollary~\ref{cor:isize} and $|g(s/m)| \ll |s|/m$.
\end{proof}
\subsection{\texorpdfstring{$H_n$}{Hn}}
Let $H_n:= \sum_{i=1}^{n} \frac{1}{i}$ be the $n$th harmonic sum. By the Euler--Maclaurin formula,
\begin{equation}\label{eq:harmonic}
H_n = \log n + \gamma + \frac{1}{2n} +O(n^{-2}).
\end{equation}
\section{Proof of Theorem~\ref{thm:between}}
We break the theorem into a proposition and a lemma.
\begin{proposition}\label{prop:between}
Uniformly for $n \ge m \ge 1$ we have
	\begin{equation}\label{eq:rhoall}
		p_{n,m} = \rho(u) \exp\left( \frac{u\xi(u)}{2m}\right) \exp\left(O\left(\frac{u\log^2(u+1)}{m^2} + \frac{1}{u} \right)\right).
	\end{equation}
	If $\sqrt{n \log n}=O(m)$ then
\begin{equation}\label{eq:2ndpartprop}
p_{n,m} = \rho(u) \exp\left( \frac{u\xi(u)}{2m}\right) \left(1 + O\left( \frac{\log (u+1)}{m}\right)\right).
\end{equation}
\end{proposition}
In \cite[Cors.~3, 5]{Manstavicius2016}, \eqref{eq:rhoall} is proved in the narrower range $n \ge m \ge n^{1/3}(\log n)^{2/3}$.
\begin{lem}\label{lem:shif}
	Uniformly for $n \ge m \ge 1$ we have
\begin{equation}\label{eq:rhoxirelation} \rho\left(u\right)\exp\left( \frac{u\xi(u)}{2m}\right) = \rho\left( u -\frac{u}{2m}\right) \exp\left(O\left(\frac{u}{m^2} + \frac{1}{m}\right)\right)
\end{equation}
and 
\begin{equation}\label{eq:m12}
\rho\left( u-\frac{u}{2m}\right) = \rho\left( \frac{n}{m+\frac{1}{2}}\right)\exp\left( O\left( \frac{u\log(u+1)}{m^2} \right)\right).
\end{equation}
\end{lem}
We prove \eqref{eq:rhoall} in \S\ref{sec:propbetween1}, and \eqref{eq:2ndpartprop} in \S\ref{sec:prelim2ndpart}--\S\ref{sec:propbetween2}. The proof of Lemma~\ref{lem:shif} is given in \S\ref{sec:lemshif}. To deduce Theorem~\ref{thm:between}, we combine \eqref{eq:2ndpartprop} and Lemma~\ref{lem:shif} if $m \ge \sqrt{n\log n}$, and \eqref{eq:rhoall} and Lemma~\ref{lem:shif} if $m\le \sqrt{n\log n}$. Here we use $\frac{u}{m^2}+\frac{1}{m} +\frac{u\log(u+1)}{m^2}\ll \frac{\log(u+1)}{m}$ when $m\ge \sqrt{n\log n}$ and $\frac{1}{u} + \frac{u}{m^2}+\frac{1}{m} +\frac{u\log(u+1)}{m^2}\ll \frac{u\log^2(u+1)}{m^2}$ when $m\le \sqrt{n\log n}$.
\subsection{Proof of the first part of Proposition~\ref{prop:between}}\label{sec:propbetween1}
Here we prove \eqref{eq:rhoall}. If $m=O(\log (u+1))$ then \eqref{eq:rhoall} follows from  \eqref{eq:pnmbound}. Hence we suppose that $m \ge C \log (u+1)$ from now on.
Similarly, if $u =O(1)$ then \eqref{eq:rhoall} is already in \eqref{eq:pnmbound}, so we suppose $u \ge C$. 
We define a function
\[ D(z) :=z^{-n} \exp\left(\sum_{i=1}^{m} \frac{z^i}{i}\right), \]
so that $x$ defined in \eqref{eq:saddledef2} solves $D'(z)=0$, and \eqref{eq:saddle} can be written as
\begin{equation}\label{eq:saddle2}
	p_{n,m}= \frac{D(x)}{\sqrt{2 \pi \lambda}} (1+O( u^{-1})).
\end{equation}
We can write $\log D(x)$ as
\[ \log D(x)=-n \log x + \sum_{j=1}^{m} \frac{1}{j}+I(m\log x)+T(m\log x),\]
see \cite[Eq.~(40)]{Manstavicius2016} for the details. 
We set
\[ \xii := m \log x\]
and define $\uu \ge 1$ as
\[ \uu := \frac{e^{\xii}-1}{\xii}.\]In this notation, \eqref{eq:saddle2} implies (using \eqref{eq:rho and i} with $\uu$ in place of $u$, and \eqref{eq:harmonic}) that
\[	p_{n,m} = \rho(\uu) \exp\left( (\uu-u)\xii\right) \exp \left(T(\xii)\right)\sqrt{ \frac{m^2 I''(\xi(\uu))}{\lambda}}\exp\left( O\left(u^{-1}+ \uu^{-1}+m^{-1}\right)\right)\]
for $u \ge C$. A short computation shows that $u=I'(\xi)=(I+T)'(\xii)$ holds by the definition of $\xi$, $\xii$ and $x$. Since $T'(t) \ge 0$ for $t \ge 0$ we have $\xii\le \xi$ and so, by the monotonicity of $I'(t)=(e^t-1)/t$, it follows that \[\uu=I'(\xii)\le I'(\xi)=u.\]
By \cite[Eq.~(22)]{Manstavicius2016},
\begin{equation}\label{eq:xidiff}
\xii = \xi  + O\left(\frac{\log (u+1)}{m}\right)
\end{equation}
uniformly in $m \ge \log u>0$. From \eqref{eq:xidiff} we see that, uniformly in $m \ge \log u>0$
\begin{equation}\label{eq:udiff}
\uu = u + O\left( \frac{u\log(u+1)}{m}\right).
\end{equation}
We will be using \eqref{eq:xidiff} and \eqref{eq:udiff} frequently.
For $m \ge \log u>0$ we have \[ \lambda = \frac{m^2}{\xi'(u)}\left(1+O\left( \frac{\log (u+1)}{m}\right)\right)\]
by \cite[Eq.~(23)]{Manstavicius2016}. Hence, since $I''(\xi(u))=1/\xi'(u)$,
\[ \sqrt{ \frac{m^2 I''(\xi(\uu))}{\lambda }}=\sqrt{ \frac{m^2}{\lambda \xi'(\uu)}}=\sqrt{\frac{\xi'(u)}{\xi'(\uu)}} \exp\left(O\left(\frac{\log(u+1)}{m}\right)\right)\]
for $m \ge C\log (u+1)>0$. Since $\xi'(t)\asymp 1/t$ and $\xi''(t) =-\xi'(t)I^{(3)}(\xi(t))/I''(\xi(t))^2 \ll 1/t^2$ for $t \ge C$,
\[ \frac{\xi'(u)}{\xi'(\uu)} =1 +O\left( \frac{u-\uu}{\uu}\right) = \exp\left(O\left( \frac{\log(u+1)}{m}\right)\right)\]
holds for $m \ge C\log (u+1)>0$ and $u \ge C$ and so
\[ \sqrt{ \frac{m^2I''(\xi(\uu))}{\lambda}}= \exp\left(O\left(\frac{\log(u+1)}{m}\right)\right)\] 
when $m \ge C\log (u+1)>0$ and $u \ge C$. At this point, we have established that
\[p_{n,m} = \rho(\uu) \exp\big((\uu-u) \xii\big) \exp \left(T(\xii)\right)\exp\bigg( O\bigg(  \frac{1}{u}+ \frac{\log(u+1)}{m}\bigg)\bigg)\]
holds in our range.
We have $\xii \le \xi \le 2\log u \le \pi m$ by Lemma \ref{lem:uniformxi} and our assumption $m \ge C\log(u+1)$. By \eqref{eq:lemT1}, \eqref{eq:xidiff} and \eqref{eq:udiff},
\[	T(\xii) = \frac{\uu\xii}{2m} + O\bigg( \frac{u\log^2(u+1)}{m^2}+\frac{\log(u+1)}{m}\bigg)= \frac{u\xi(u)}{2m} + O\bigg( \frac{u\log^2(u+1)}{m^2} + \frac{\log(u+1)}{m}\bigg)\]
in this range. We now have
\[	p_{n,m} = \exp\bigg(\frac{u\xi(u)}{2m}\bigg)\rho(\uu) \exp\big( (\uu-u)\xii\big) \exp\bigg( O\bigg(\frac{1}{u}+ \frac{u\log^2(u+1)}{m^2}\bigg)\bigg) \]
for $m \ge C\log (u+1)$, $u \ge C$.
Let $r(t)=-\rho'(t)/\rho(t)$ be the negative of the logarithmic derivative of $\log \rho$. We have \[\rho(\uu)=\rho(u)\exp\left(\int_{\uu}^{u}r(t){\mathrm d}t\right).\]
By integration by parts, we may write the integral in the exponent as
\[	\int_{\uu}^{u}r(t){\mathrm d}t
	= (u-\uu)r(u) - \int_{\uu}^{u} (t-\uu)r'(t) {\mathrm d}t. \]
By \cite[Lem.~3.7]{LaBreteche2005}, the estimates
\begin{equation}\label{eq:LaBrTe}
r(v) = \xi(v) + O\left(v^{-1}\right), \qquad  r'(v)  = \xi'(v)+O\left(v^{-2}\right)
\end{equation}
hold uniformly for $v \ge 1$. Hence, using \eqref{eq:xidiff} and \eqref{eq:udiff},
\begin{align*}
\rho(\uu) \exp( (\uu-u)\xii) &= \rho(u) \exp((u-\uu) (r(u)-\xi(u)+\xi(u)-\xii) \\
& \qquad - \int_{\uu}^{u}(t-\uu)(r'(t)-\xi'(t)+\xi'(t)){\mathrm d}t )\\
&=\rho(u) \exp\bigg( O\bigg(\frac{1}{u} +  \frac{u\log^2(u+1)}{m^2}\bigg)-\int_{\uu}^{u} (t-\uu)\xi'(t){\mathrm d}t\bigg).
\end{align*}
To conclude the proof, it remains to bound $\int_{\uu}^{u}(t-\uu)\xi'(t){\mathrm d}t$. Since $\xi'(t) \ll 1/t$,
\[ \int_{\uu}^{u}(t-\uu)\xi'(t){\mathrm d}t\ll \int_{\uu}^{u} \frac{t-\uu}{t} {\mathrm d}t = \uu \left( \frac{u}{\uu}-1-\log \left(\frac{u}{\uu}\right)\right) \ll \uu \frac{(u-\uu)^2}{\uu^2} \ll \frac{u\log^2(u+1)}{m^2}\]
if $m \ge C\log (u+1)$, using \eqref{eq:xidiff} and \eqref{eq:udiff}. The proof of \eqref{eq:rhoall} is completed.
\subsection{Preliminary lemmas for second part of Proposition~\ref{prop:between}}\label{sec:prelim2ndpart}
The following consequence of Cauchy's integral formula is implicit in the proofs of \cite[Thms.~2,\,4]{Manstavicius2016}. 
\begin{lem}\label{lem:pnmstart}
	We have
	\[ p_{n,m} = \frac{\exp\big(H_m -\gamma\big)}{m}	 \frac{1}{2\pi i} \int_{-\xi-im\pi}^{-\xi+im\pi} e^{us} \hat{\rho}(s) e^{T(-s)}{\mathrm d}s.\]
\end{lem}
The following lemma is implicit in the proof of \cite[Thm.~4]{Manstavicius2016}.
\begin{lem}\label{lem:rhoonly}
	Suppose $\sqrt{n\log n}=O(m)$. Then
	\[\frac{1}{2\pi i} \int_{-\xi-im\pi}^{-\xi+im\pi} e^{us} \hat{\rho}(s){\mathrm d}s = \rho(u) +O\left( \frac{e^{-(u-1)\xi}}{n}\right).\]
\end{lem}
A variant of the next lemma is implicit in the proof of \cite[Thm.~2]{Manstavicius2016}.
\begin{lem}\label{lem:rhohigh}
	Suppose $\sqrt{n\log n}=O(m)$ and $1+u\xi \le A \le m\pi$. Then
\[ \frac{1}{2\pi i} \int_{-\xi+iA}^{-\xi+im\pi} e^{us} \hat{\rho}(s) \left(e^{T(-s)-T(\xi)}-1\right){\mathrm d}s \ll \frac{e^{-u\xi}\log(u+1)}{m}. \]
The same bound holds if we integrate from $-\xi-im\pi$ to $-\xi-iA$.
\end{lem}
\begin{proof}
We have $\hat{\rho}'(s) = \rho(s) (e^{-s}-1)/s$. Integration by parts shows that the integral is equal to 
	\begin{multline*}
	\frac{1}{2\pi i} \frac{e^{us}}{u}\hat{\rho}(s) \left(e^{T(-s)-T(\xi)}-1\right) \Big|^{s=-\xi+im\pi}_{s=-\xi+iA} \\-\frac{1}{2\pi i} \int_{-\xi+iA}^{-\xi+im\pi} \frac{e^{us}}{u} \hat{\rho}(s)\left( \frac{e^{-s}-1}{s} \left(e^{T(-s)-T(\xi)}-1\right) -T'(-s) e^{T(-s)-T(\xi)}\right){\mathrm d}s.
\end{multline*}
	By Corollary~\ref{cor:t}, $T(-s)-T(\xi)$ is bounded in our range of integration. In fact, $T(-s)-T(\xi) \ll (e^{\xi}+|\Im s|)/m$. Hence, the third case of Lemma~\ref{lem:i bounds} shows that
	\[ \frac{1}{2\pi i} \frac{e^{us}}{u}\hat{\rho}(s) \left(e^{T(-s)-T(\xi)}-1\right) \Big|^{s=-\xi+im\pi}_{s=-\xi+iA} \ll \frac{e^{-u\xi}}{n}\]
which is acceptable. As for the integral, we rearrange it using the definition of $T$:
	\begin{align*}
		\frac{1}{2\pi i} &\int_{-\xi+iA}^{-\xi+im\pi} \frac{e^{us}}{u} \hat{\rho}(s)\left( \frac{e^{-s}-1}{s} \left(e^{T(-s)-T(\xi)}-1\right)-T'(-s) e^{T(-s)-T(\xi)}\right){\mathrm d}s \\
		&= 	\frac{1}{2\pi i} \int_{-\xi+iA}^{-\xi+im\pi} \frac{e^{us}}{u} \frac{e^{-s}-1}{s} \hat{\rho}(s)\left( e^{T(-s)-T(\xi)}\frac{\frac{-s}{m}}{1-e^{s/m}}-1\right){\mathrm d}s.
	\end{align*}
By Corollary~\ref{cor:t} and the Taylor expansion $z/(1-e^{-z})=1+z/2+O(z^2)$, 
\[ e^{T(-s)-T(\xi)}\frac{\frac{-s}{m}}{1-e^{s/m}}-1 \ll \frac{|s|^2}{m^2} + \frac{e^{\xi}}{m}\]
in our range of integration. By the triangle inequality, the last integral is
\begin{multline*} \ll \frac{\log (u+1)}{e^{u\xi}} \int_{-\xi+iA}^{-\xi+im\pi} \frac{|\hat{\rho}(s)|}{|s|}\left( \frac{|s|^2}{m^2}+\frac{e^{\xi}}{m}\right)  |{\mathrm d}s|\\
 \ll \frac{\log (u+1)}{e^{u\xi}} \int_{-\xi+iA}^{-\xi+im\pi} \frac{1}{|s|^2}\left( \frac{|s|^2}{m^2}+\frac{e^{\xi}}{m}\right)  |{\mathrm d}s| \ll \frac{\log (u+1)}{m e^{u\xi}} 
\end{multline*}
where in the second inequality we used the third case of Lemma~\ref{lem:i bounds}.
\end{proof}
\begin{lem}\label{lem:abso}
	Suppose $0 \le B \le A$. Then, for every $k \ge 0$,
	\[\frac{1}{2\pi i} \int_{-\xi+iB}^{-\xi+iA} \left| (s+\xi)^k e^{us} \hat{\rho}(s)\right|| {\mathrm d}s| \ll_k \rho(u) \big( u^{-\frac{k}{2}} e^{-cuB^2} + A^{k+1}\sqrt{u}\exp\big(-\frac{u}{\xi^2+\pi^2}\big)\big).\]
	The same is true if we integrate from $-\xi-iA$ to $-\xi-iB$. In particular, for $B=0$ and $A\le \exp(c_k u/(\log (u+1))^2)$,
		\[\frac{1}{2\pi i} \int_{-\xi-iA}^{-\xi+iA} \left| (s+\xi)^k e^{us} \hat{\rho}(s)\right|\left| {\mathrm d}s\right| \ll_k \rho(u)u^{-\frac{k}{2}}.\]
\end{lem}
\begin{proof}
	Using the first two parts of Lemma~\ref{lem:i bounds}, the triangle inequality shows that the integral is
	\begin{multline*}
		\ll e^{-u\xi +I(\xi)} \int_{B}^{A} |t|^k \left(\exp(-cut^2)+\exp\big(-\frac{u}{\xi^2+\pi^2}\big)\right) {\mathrm d}t\\ \ll_k e^{-u\xi +I(\xi)}  \bigg( u^{-\frac{k+1}{2}} \exp(-cuB^2)+ A^{k+1}\exp\big(-\frac{u}{\xi^2+\pi^2}\big)\bigg),
	\end{multline*}
	and we now use \eqref{eq:rho and i}.
\end{proof}
\begin{lem}\label{lem:rholin}
	Fix $k\ge 0$. Let $\mu_m:=m!/((m/2)!2^{m/2})$ for even $m$. Suppose $\sqrt{n\log n}=O(m)$ and $C_k\sqrt{\log(u+1)/u}\le A \le \exp(c_ku/(\log (u+1))^2)$. Then, for even $k$,
	\[\frac{1}{2\pi i} \int_{-\xi-iA}^{-\xi+iA} (s+\xi)^ke^{us} \hat{\rho}(s){\mathrm d}s = (1+O_k(u^{-1}))\rho(u) \mu_k  (-1)^{\frac{k}{2}} \xi'^{\frac{k}{2}}\]
	and for odd $k$,
	\[\frac{1}{2\pi i} \int_{-\xi-iA}^{-\xi+iA} (s+\xi)^ke^{us} \hat{\rho}(s){\mathrm d}s = (1+O_k(u^{-1}))\rho(u) \mu_{k+3} (-1)^{\frac{k+1}{2}} \xi'^{\frac{k+3}{2}} \frac{I^{(3)}(\xi)}{6}.\]
\end{lem}
\begin{proof}
The contribution of $A \ge |\Im s| \ge C_k\sqrt{\log (u+1)/u}$ is acceptable by Lemma~\ref{lem:abso}.
	We may now assume that $A=C_k\sqrt{\log (u+1)/u}$. Recall $\hat{\rho}(s)=\exp(\gamma+I(-s))$. We may Taylor-expand $I(\xi+it)$ using Corollary~\ref{cor:isize}, obtaining that
	\begin{align*}
		\frac{1}{2\pi i} &\int_{-\xi-iA}^{-\xi+iA} (s+\xi)^k e^{us} \hat{\rho}(s){\mathrm d}s \\
		&= i^{k} \frac{e^{\gamma-u\xi+I(\xi)}}{2\pi}\int_{-A}^{A} t^k \exp\left( -\frac{t^2}{2} I''(\xi) +i\frac{t^3}{6} I^{(3)}(\xi)+\frac{t^4}{24}I^{(4)}(\xi)+O_k\left(u|t|^5\right) \right){\mathrm d}t\\
		&= i^{k}\frac{e^{\gamma-u\xi+I(\xi)}}{2\pi} \int_{-A}^{A} t^k \exp\big(-\frac{t^2}{2} I''(\xi)\big) \big(1+i \frac{t^3}{6} I^{(3)}(\xi) + \frac{t^4}{24}I^{(4)}(\xi) - \frac{t^6}{72}I^{(3)}(\xi)^2\big)\\
		&\qquad \qquad \cdot \left(1+O_k\left( u|t|^5 + u^3 |t|^9\right)\right) {\mathrm d}t.
	\end{align*}
Substituting $t^2 I''(\xi)=v^2$, recalling $I''(\xi(u))=1/\xi'(u)$ and using  \eqref{eq:rho and i}, the last expression can be written as
\begin{multline*}= \rho(u)(1+O_k(u^{-1})) i^k \xi'^{k/2} \frac{1}{\sqrt{2\pi}}\int_{-A/\sqrt{\xi'}}^{A/\sqrt{\xi'}} v^k \exp(-v^2/2)\\ \cdot \bigg(1+i \frac{v^3 \xi'^{3/2}}{6} I^{(3)}(\xi) + \frac{v^4 \xi'^2 }{24}I^{(4)}(\xi) - \frac{v^6 \xi'^3}{72}I^{(3)}(\xi)^2 \bigg) \bigg( 1+ O_k\bigg( \frac{|v|^5 +|v|^9}{u^{3/2}} \bigg)\bigg) {\mathrm d}v.
\end{multline*}
To conclude we apply estimates of gaussian integrals: $\int_{-R}^{R} v^{2k+1}\exp(-v^2/2){\mathrm d}v$ is $0$, $\int_{\RR}|v|^k \exp(-v^2/2){\mathrm d}v \ll_k 1$ and $\int_{-R}^{R}v^{2k}\exp(-v^2/2){\mathrm d}v =\sqrt{2\pi}\mu_{2k} + O_k(\exp(-R^2/4))$.
\end{proof}
\subsection{Proof of the second part of Proposition~\ref{prop:between}}\label{sec:propbetween2}
Here we prove \eqref{eq:2ndpartprop} using the material in \S\ref{sec:prelim2ndpart}.
If $1+u\xi > m \pi$ the result is already included in the first part of Proposition~\ref{prop:between}, so we assume from now on that $1+u\xi \le m \pi$.
From Lemma~\ref{lem:pnmstart} we have
\[ p_{n,m} = \frac{\exp\big(H_m -\gamma\big)}{m} e^{T(\xi)} 	 \frac{1}{2\pi i} \int_{-\xi-im\pi}^{-\xi+im\pi} e^{us} \hat{\rho}(s) e^{T(-s)-T(\xi)}{\mathrm d}s.\]
We have $\exp(H_m-\gamma)/m  =1+1/(2m)+O(m^{-2})$
by \eqref{eq:harmonic}, and
\[ e^{T(\xi)} = e^{\frac{u\xi}{2m}} \left( 1- \frac{\xi}{2m}+ O\left(\frac{u\log^2(u+1)}{m^2}\right)\right)\]
by \eqref{eq:lemT1}. Hence,
\[ p_{n,m} = e^{\frac{u\xi}{2m}} \left(1 + \frac{1-\xi}{2m} + O\left(\frac{u\log^2(u+1)}{m^2}\right)\right)  \frac{1}{2\pi i} \int_{-\xi-im\pi}^{-\xi+im\pi} e^{us} \hat{\rho}(s) e^{T(-s)-T(\xi)}{\mathrm d}s.\]
We separate the integral into 3 parts, $S_1+S_2+S_3$, where
\begin{align*}
	S_1 &=  \frac{1}{2\pi i} \int_{-\xi-im\pi}^{-\xi+im\pi} e^{us} \hat{\rho}(s) {\mathrm d}s,\\
	S_2 &=  \frac{1}{2\pi i} \int_{-\xi-i(1+u\xi)}^{-\xi+i(1+u\xi)} \left(e^{T(-s)-T(\xi)}-1\right)e^{us} \hat{\rho}(s) {\mathrm d}s,\\
	S_3 &=  \frac{1}{2\pi i} \left(\int_{-\xi-im\pi}^{-\xi-i(1+u\xi)}+\int_{-\xi+i(1+u\xi)}^{-\xi+im\pi}\right) \left(e^{T(-s)-T(\xi)} - 1 \right)e^{us} \hat{\rho}(s) {\mathrm d}s.
\end{align*}
The integral $S_1$ was estimated in Lemma~\ref{lem:rhoonly} and it gives the main term $\rho(u)$ as well an absolute error of size $\ll e^{-(u-1)\xi}/n \ll e^{-u\xi} \log(u+1)/m$. The integral $S_3$ was estimated in Lemma~\ref{lem:rhohigh} and it contributes $\ll e^{-u\xi}\log(u+1)/m$. 
We now study $S_2$. We use Lemma~\ref{lem:Tderivs} to Taylor-expand $e^{T(-s)-T(\xi)}-1=e^{T(\xi-it)-T(\xi)}-1$ at $0$:
\begin{multline}
e^{T(-s)-T(\xi)}-1 = -T'(\xi)(s+\xi) + \frac{(s+\xi)^2}{2}\left( T''(\xi)+T'(\xi)^2\right)\\
-(T^{(3)}(\xi)+3T'(\xi)T''(\xi)+(T'(\xi))^3) \frac{(s+\xi)^3}{6}+O\left( \frac{u\log(u+1)}{m} |s+\xi|^4\right)
\end{multline}
for $|\Im s| \le 1+u\xi$.
Applying Lemma~\ref{lem:abso} with $k=4$, $B=0$ and Lemma~\ref{lem:rholin} with $k=1,2,3$ and collecting the terms gives
\begin{multline}
	S_2 = \rho(u)\bigg(\frac{T'(\xi)(1+O(\log (u+1)^{-1}))-T''(\xi)(1+O(\log(u+1)^{-1}))}{2u} \\
 + O\big(   \frac{\log (u+1)}{n}+\frac{u\log^2(u+1)}{m^2}\big)\bigg).
 \end{multline}
All in all, since $T'(\xi)$ and $T''(\xi)$ are both $\frac{u\log(u+1)}{2m} (1+o_{u\to\infty}(1))$ in our range by Lemma \ref{lem:Tderivs}, 
\[ p_{n,m} =\exp\big(\frac{u\xi}{2m}\big) \rho(u)\bigg( 1-\frac{\log(u+1)}{2m}(1+o_{u\to\infty}(1)) + O\bigg(\frac{u\log^2 (u+1)}{m^2}\bigg)\bigg). \]
We are done, as we established, in stronger form, the second part of Proposition~\ref{prop:between}.
\subsection{Proof of Lemma~\ref{lem:shif}}\label{sec:lemshif}
We first show \eqref{eq:rhoxirelation}. Let $r(t)=-\rho'(t)/\rho(t)\ge 0$. We have
\[	\rho(u)= \rho\bigg( u-\frac{u}{2m}\bigg) \exp\bigg( -\int_{u-\frac{u}{2m}}^{u} r(t){\mathrm d}t\bigg).\]
If $u$ is bounded then the bound $r(t)\ll 1$ for $t \ll 1$ finishes the proof. If $u \ge C$ we may differentiate $r$ (it is differentiable for $t \ge 2$) and obtain 
\[ \rho(u) = \rho\bigg( u-\frac{u}{2m}\bigg) \exp\bigg( -\frac{ur(u)}{2m} + \int_{u-\frac{u}{2m}}^{u} \bigg(u-\frac{u}{2m}-t\bigg)r'(t) {\mathrm d}t\bigg)\]
by integration by parts. To conclude, we use \eqref{eq:LaBrTe} to find
\[	-\frac{ur(u)}{2m} + \int_{u-\frac{u}{2m}}^{u} \bigg(u-\frac{u}{2m}-t\bigg)r'(t) {\mathrm d}t = -\frac{u\xi(u)}{2m} + O\bigg( \frac{u}{m^2} + \frac{1}{m}\bigg).\]
It remains to show \eqref{eq:m12}. Set $t_1 = n/(m+1/2)$ and $t_2=u-u/(2m)$. It suffices to bound $\log (\rho(t_2)/\rho(t_1))$. Observe $t_1\ge t_2$. We have
\[ \log \rho(t_2)-\log \rho(t_1) = \int_{t_2}^{t_1} r(t)\mathrm{d}t \le (t_1-t_2) \max_{t\in [t_2,t_1]} r(t) \ll \frac{n}{m^3} \log(u+1) = \frac{u\log(u+1)}{m^2} \]
since $r(t)\ll \log(t+1)$ by \eqref{eq:LaBrTe} and Lemma~\ref{lem:xi size}. This finishes the proof.
\section{Proof of Theorem \ref{thm:mediumm} and Corollary \ref{cor:lower}}
Let \[ d_{m,m}=-\frac{1}{m}\sum_{j=2}^{m}\frac{1}{j}, \qquad d_{m,i}=\frac{\Gamma\left(i+\frac{i}{m}\right)}{(m-i)i!\Gamma\left(1+\frac{i}{m}\right)}\quad  (1 \le i \le m-1).\]
\begin{lem}\label{lem:transition}
	The estimate 
\[	p_{n,m} = \frac{\exp\big( - u \log n + u + \sum_{i=1}^{m} d_{m,i} n^{1-\frac{i}{m}} + E\big)}{\sqrt{2\pi n m}}  \bigg(1+O\bigg(\frac{1}{\max\{\log u, n^{1/m}\}}\bigg)\bigg)\]
	holds uniformly for $n>m \ge 1$, where $E$ is a quantity satisfying
	\[ E \ll \frac{n^{-\frac{1}{m}}}{1-n^{-\frac{1}{m}}} + \frac{\frac{m}{n}}{1-\left( \frac{m}{n}\right)^{\frac{1}{m}}}.\]
\end{lem}
Lemma \ref{lem:transition} is a minor improvement on Theorem 1 of \cite{Manstavicius2016}, which treats $m\le \log n$.
\begin{proof}
Recall the definitions of $x$ and $\lambda$ given in \eqref{eq:saddledef2} and \eqref{eq:lambedadef}. Following the proof of Theorem 1 in \cite{Manstavicius2016}\footnote{We correct a few typos. The definition of $D(x)$ in p.~21 of \cite{Manstavicius2016} should be replaced by the one given in p.~4 of that paper. The left-hand side of their equation (48) should be ``$\log D(x)$''. In the first displayed equation in p.~22, the sum of $h_N z^N$ should start at $N=-r$ and not $N=-r+1$. In the second displayed equation in p.~23, the factor $n!$ that appears twice should be omitted.} we have, borrowing the notation of \cite[Sec.~5]{Manstavicius2016},
\[ p_{n,m} = \frac{1}{\sqrt{2\pi \lambda}}\exp\big(-u\log n + u +\sum_{i=1}^{m} d_{m,i} n^{1-\frac{i}{m}} +E\big)\big( 1+ O\big(u^{-1}\big)\big)\]
for $E= R\big(n^{-\frac{1}{m}}\big)$ where
\[ R(z): = \sum_{i=1}^{\infty} h_i z^i -n \sum_{i=m+1}^{\infty} b_i z^i\]
for certain coefficients $h_i$ and $b_i$ defined in Lemma 13 of \cite{Manstavicius2016} and estimated in Lemma 15 of \cite{Manstavicius2016}. In particular, 
\[ h_i = \frac{i+m}{i}b_{i+m}\]
for $i \ge 1$ and 
\[ b_i \ll \frac{m^{\frac{i}{m}}}{i}\]
for $i \ge 1$.
In the first displayed equation of Lemma 15 of \cite{Manstavicius2016} it is shown that $b_i \ll i^{\frac{i}{m}-1}/m$ for $m+1 \le i \le 2m-1$. It implies
\[	b_i \ll \frac{i-m}{m} \]
in the same range. Hence
\begin{align*} R(z)&\ll \sum_{i=1}^{m-1} \frac{i+m}{m}|z|^i + \sum_{i=m}^{\infty} \frac{m}{i}( m^{\frac{1}{m}}|z|)^i + n\bigg(\sum_{i=m+1}^{2m-1} \frac{i-m}{m}|z|^i + \sum_{i=2m}^{\infty}\frac{1}{i}(m^{\frac{1}{m}}|z|)^i\bigg)\\
&\ll \sum_{i=1}^{m-1} |z|^{i} + \sum_{i=m}^{2m-1} \bigg( \frac{n(i-m)}{m} + m^{\frac{i}{m}}\bigg) |z|^i + \sum_{i=2m}^{\infty} \frac{n}{i} (m^{\frac{1}{m}}|z|)^i.
\end{align*}
It follows that
\[ R(n^{-\frac{1}{m}}) \ll \frac{n^{-\frac{1}{m}}}{1-n^{-\frac{1}{m}}} + \frac{n^{-\frac{1}{m}}}{m(1-n^{-\frac{1}{m}})^2}+ \frac{\frac{m}{n}}{1-(\frac{m}{n})^{\frac{1}{m}}}\ll \frac{n^{-\frac{1}{m}}}{1-n^{-\frac{1}{m}}} + \frac{\frac{m}{n}}{1-(\frac{m}{n})^{\frac{1}{m}}}.\]
To conclude, we input the estimates for $\lambda/(nm)$ given in Lemma~\ref{lem:xlam}.
\end{proof}

\subsection{Proof of Theorem~\ref{thm:mediumm}}
Let
\[ S:=\sum_{i=1}^{m-1} d_{m,i} n^{1-\frac{i}{m}}.\]
In view of Lemma~\ref{lem:transition} we have
\begin{equation}\label{eq:pnmS}
	p_{n,m} = \left(\frac{e}{n}\right)^u \exp\left( S + O(\log n)\right)
\end{equation}
when $m=O(\log n)$ and it remains to estimate $S$. We have
\[\frac{1}{m-i} = \frac{1}{m} \left(1+O\left( \frac{i}{m-i}\right)\right), \qquad \frac{1}{\Gamma\left(1+\frac{i}{m}\right)} = 1+O\left( \frac{i}{m}\right)\]
uniformly for $1 \le i \le m-1$. The estimate $\Gamma(x+y)/\Gamma(x) \in [x(x+y)^{y-1},x^y]$ for $y \in (0,1)$ \cite[Eq.~(7)]{Wendel1948} implies
\[\frac{\Gamma\left(i+\frac{i}{m}\right)}{i!} =\frac{1}{i} \frac{\Gamma\left( i +\frac{i}{m}\right)}{\Gamma(i)}= \frac{i^{\frac{i}{m}}}{i} \left(1+O\left( \frac{1}{m}\right)\right).\]
Hence,
\[d_{m,i}=\frac{i^{\frac{i}{m}}}{mi}\left(1+O\left( \frac{i}{m-i}\right)\right).\]
Since $i^{i/m}=1+O(\tfrac{i\log i}{m})$ if $i<m/\log (m+1)$, we may write $S$ as $S= u(S_1+O(S_2+S_3))$ where
\[S_1 =\sum_{i=1}^{m-1} \frac{n^{-\frac{i}{m}}}{i},\,\quad S_2=\sum_{1 \le i < \frac{m}{\log (m+1)}} \frac{n^{-\frac{i}{m}}}{i} \frac{i\log(i+1)}{m},\, \quad S_3=\sum_{\frac{m}{\log(m+1)} \le i \le m-1}  (i/n)^{\frac{i}{m}}\frac{m}{i}.\]
We have
\[ S_1 =  -\log(1-n^{-\frac{1}{m}}) + O\left( \frac{1}{nm}\frac{1}{1-n^{-\frac{1}{m}}}\right)\]
by bounding the tail of the Taylor series of $-\log(1-n^{-1/m})$ by a geometric series with ratio $n^{-1/m}$. Similarly,
\[ S_2 \ll \frac{1}{m} \frac{1}{n^{\frac{1}{m}}-1}\ll \frac{n^{-\frac{1}{m}}}{m}\]
because the contribution of $i \in [2^k,2^{k+1})$ to $S_2$ is $\ll (k+1)n^{-2^k/m}/(m(1-n^{-1/m}))$ and $\sum_{k \ge 0}(k+1)n^{-2^k/m}\ll n^{-1/m}$ when $m\ll \log n$. As for $S_3$,
\[S_3\ll m \sum_{\frac{m}{\log(m+1)}\le i \le m-1}\frac{1}{i} \max_{\frac{m}{\log(m+1)}\le i \le m-1} (i/n)^{i/m} \ll  m\log \log (m+2) n^{-\frac{1}{\log(m+1)}} \]
since $i\mapsto (i/n)^{i/m}$ decreases for $i \le n/e$. Both $S_3$ and the error term in $S_1$ are dominated by our bound for $S_2$. It follows that 
\[ S =  u \big( -\log\big(1-n^{-\frac{1}{m}}\big) + O\big( \frac{n^{-\frac{1}{m}}}{m}\big)\big).\]
The error $O(\log n)$ in \eqref{eq:pnmS} is absorbed in the error term in $S$ when $m>1$, and we are done.
\subsection{Proof of Corollary~\ref{cor:lower}}
We first consider $n \ge m > n/2$. For $m=n$, $p_{n,n}=\rho(1)=1$, so we may assume $m\le n-1$. We have exact formulas: $p_{n,m}=1-\sum_{i=m+1}^{n}1/i$ and $\rho(u)=1-\log u$. Hence
\[p_{n,m}-\rho(u)=  \sum_{i=m+1}^{n}( \log i - \log(i-1) - \frac{1}{i}) \asymp \sum_{i=m+1}^{n} \frac{1}{i^2} \asymp \int_{m}^{n} \frac{{\mathrm d}t}{t^2}.\]
We are done since, in our range,
\[\int_{m}^{n} \frac{{\mathrm d}t}{t^2} =\frac{1}{m}-\frac{1}{n}\asymp\frac{n-m}{m^2} =\frac{u-1}{m}\asymp \frac{\log u}{m}.\]
We now suppose $n/2 \ge m \ge C \log n$. We shall need the lower bound \cite[Thm.~A.1]{Gorodetsky2022}
\begin{equation}\label{eq:lowerpnm}
p_{n,m} \ge \rho(u)\bigg(1+\frac{cu\log u}{m}\bigg)
\end{equation}
which holds uniformly for $n/2 \ge m \ge 1$. We may assume $n\ge C$, since for bounded $n$ we just want to show $B_2\rho(u) \ge p_{n,m}\ge B_1\rho(u)$ for constants $B_2\ge B_1>1$ which follows from \eqref{eq:lowerpnm} and \eqref{eq:pnmbound}. If $u$ is sufficiently large then it follows from Proposition~\ref{prop:between} that, using the same definition for $A$ as in the statement of the corollary,
\[	A \asymp  \frac{u\xi(u)}{m}\asymp  \frac{u\log u}{m} \asymp u \log \left( 1+ \frac{\log u}{m}\right) \]
as needed. If $u=O(1)$ then the same argument establishes $A \le Cu\log ( 1+ (\log u)/m)$ and a matching lower bound in this range follows from \eqref{eq:lowerpnm}.

Finally we suppose $m=O(\log n)$. From \eqref{eq:db} we have
	\begin{equation}\label{eq:aux}
	\left(\frac{e}{n}\right)^u \frac{1}{\rho(u)} = \exp\left( u \left( \log\left( \frac{\log u}{m}\right) + O\left( \frac{\log \log u}{\log u}\right) \right)\right).
	\end{equation}
From Theorem~\ref{thm:mediumm} and \eqref{eq:aux}, 
\[ u^{-1}A = \log\left( \frac{\log n}{m}\right) - \log\left(1-n^{-\frac{1}{m}}\right) +O\left(\frac{\log \log n}{\log n}\right).\]
If $n$ is sufficiently large and $1\le m\le C\log n$,
\[\log\bigg( \frac{\log n}{m}\bigg) - \log(1-n^{-\frac{1}{m}}) = \log \bigg( \frac{(\log n)/m}{1-e^{-(\log n)/m}}\bigg) \asymp  \log\bigg(1+\frac{\log n}{m}\bigg)\]
since $t/(1-e^{-t}) \ge 1+c$ when $t=(\log n)/m\ge 1/C$. This finishes the proof.
\begin{remark}
When $m=O(\log n)$ the proof above shows more: setting $t:=(\log n)/m$,
	\[ A= u \left( \log \frac{t}{1-e^{-t}} +O\left ( \frac{\log\log (n+2)}{\log (n+2)}\right)\right)\]
holds, where $A$ is as in the statement of Corollary~\ref{cor:lower}.
\end{remark}
\section{Proofs of results in the polynomial setting}
We recall we can write $G_q$ as
\begin{equation}\label{eq:Gqnotation}
	G_q(z) =\exp\left(\sum_{i=m+1}^{\infty} \frac{a_iz^i}{i}\right)
\end{equation}
where $a_{i}$ are nonnegative numbers, depending on $q$, that are described in \cite[Lem.~2.1]{Gorodetsky2022} and satisfy 
\begin{equation}\label{eq:aisize}
a_i \ll \min\{q^{-\lceil i/2\rceil}, q^{m-i}\}.
\end{equation}
\subsection{Proof of Theorem~\ref{thm:smallupnmq}}
We suppose that $n \ge m \ge \sqrt{10n\log n}$. This guarantees $2^{m/3} > 1+ m(\log 2)/3 \ge 1 + 2u\log u \ge 1+u\xi = e^{\xi}$, where we used Lemma~\ref{lem:uniformxi} in the last inequality. Hence
\begin{equation}\label{eq:eximbnd}
e^{\xi/m} < \sqrt[3]{2}.
\end{equation}
An application of Cauchy's integral formula allows us to express $p_{n,m}$ as
\begin{equation}\label{eq:pnmcauchy}
p_{n,m} =  \frac{\exp(H_m-\gamma)}{m} \frac{1}{2\pi i}\int_{-\xi-im\pi}^{-\xi+im\pi} \hat{\rho}(s)e^{us} e^{T(-s)}{\mathrm d}s,
\end{equation}
see \cite[\S4]{Manstavicius2016}. In the same way,
\begin{equation}\label{eq:pnmqcauchy}
p_{n,m,q} = \frac{\exp(H_m-\gamma)}{m} \frac{1}{2\pi i}\int_{-\xi-im\pi}^{-\xi+im\pi} \hat{\rho}(s)e^{us} e^{T(-s)} G_q(e^{-s/m}) {\mathrm d}s,
\end{equation}
see e.g.~\cite[\S4]{Gorodetsky2022}. Since $G_q$ has radius of convergence equal to $q$, we must ensure that $e^{\xi/m}< q$ in order for the last integral to be valid, and this holds by \eqref{eq:eximbnd}. By taking a linear combination of  \eqref{eq:pnmcauchy} and \eqref{eq:pnmqcauchy}, and using \eqref{eq:harmonic} we have
\[ p_{n,m,q}-G_q(e^{\xi/m})p_{n,m} = G_q(e^{\xi/m}) \frac{1+O(m^{-1})}{2\pi i}X \]
for
\[ X:= \int_{-\xi-im\pi}^{-\xi+im\pi} \hat{\rho}(s)e^{us} e^{T(-s)} \bigg(\frac{G_q(e^{-s/m})}{G_q(e^{\xi/m})}-1\bigg) {\mathrm d}s.\]
Since $p_{n,m} \gg \rho(u)$ for $m \ge \sqrt{10n\log n}$ by \eqref{eq:pnmbound}, it follows that
\[ \frac{p_{n,m,q}}{p_{n,m}}=G_q(e^{\xi/m})\bigg(1+O\bigg(\frac{|X|}{\rho(u)}\bigg)\bigg).\]
We must show $X \ll \rho(u) \log(u+1)/(mq^{\lceil(m+1)/2\rceil})$. To bound $X$, we consider separately the contribution of $|\Re s| \le 1+u\xi$ and $m\pi \ge |\Re s| \ge 1+u\xi$. We start with $|\Re s | \le 1+u\xi$.
Let 
\[ H(s):= \frac{G_q(e^{-s/m})}{G_q(e^{\xi/m})}-1, \qquad H_T(s):= H(s) e^{T(-s)}\]
so that the integrand in $X$ is $\hat{\rho}(s) e^{us} H_T(s)$. We Taylor-expand $H_T(s)$ at $s=-\xi$, where it attains $0$: $H_T(s) = (s+\xi)b_1 + O(|s+\xi|^2 b_2)$ for 
\[b_1= H_T'(-\xi) \qquad \text{and} \qquad b_2=\max_{|t| \le 1+u\xi} |H_T''(-\xi+it)|.\]
Applying Lemma~\ref{lem:abso} with $k=2$, $B=0$ and Lemma~\ref{lem:rholin} with $k=1$, it follows that
\[ \int_{-\xi-i(1+u\xi)}^{-\xi+i(1+u\xi)} \hat{\rho}(s)e^{us} H_T(s) {\mathrm d}s \ll \left(|b_1|+b_2\right)\frac{\rho(u)}{u}.\]
We now consider the contribution of $\pi m \ge |\Re s| \ge 1+u\xi$. We focus on $m\pi \ge \Re s \ge 1+u\xi$, and negative $\Re s$ is handled the same say.
We have $\hat{\rho}(s) = 1/s + O(u\log (u+1) /|s|^2)$ in this range by the third part of Lemma~\ref{lem:i bounds}. The contribution of $O(u\log (u+1)/|s|^2)$ is
\[ \ll \int_{-\xi+i(1+u\xi)}^{-\xi+i\pi m}|e^{us} \frac{u\log(u+1)}{s^2} H_T(s)| |{\mathrm d}s|\ll b_3e^{-u\xi} \] 
where \[ b_3 = \max_{|t| \le \pi m} |H_T(-\xi+it)|.\]
We study the last piece of the integral,
\[ \int_{-\xi+i(1+u\xi)}^{-\xi+i\pi m} \frac{e^{us}}{s} e^{T(-s)} H(s) {\mathrm d}s .\]
We write $e^{T(-s)}$ as $1+(e^{T(-s)}-1)$.
When $s=-\xi+it$ for $1+u\xi \le t \le \pi m$ we have $e^{T(-s)}-1=O(|s|/m + |s|^2/m^2)$ by Corollary~\ref{cor:t}. Applying the triangle inequality we get
\[ \int_{-\xi+i(1+u\xi)}^{-\xi+i\pi m} \frac{e^{us}}{s} (e^{T(-s)}-1) H(s){\mathrm d}s\ll b_4 \int_{-\xi+i(1+u\xi)}^{-\xi+i\pi m}\ \bigg| \frac{e^{us}}{s}\bigg| \bigg(\frac{|s|}{m}+\frac{|s|^2}{m^2}\bigg) |{\mathrm d}s| \ll b_4 e^{-u\xi} \]
where
\[ b_4 = \max_{|t| \le \pi m} |H(-\xi+it)|.\]
We estimate
\begin{equation}\label{eq:finalboss} \int_{-\xi+i(1+u\xi)}^{-\xi+i\pi m} \frac{e^{us}}{s} H(s) {\mathrm d}s .
\end{equation}
Using the notation in \eqref{eq:Gqnotation} and the bound in \eqref{eq:aisize}, it follows that
\begin{equation}\label{eq:H}
	\begin{split}
H(s) &= \log G_q(e^{-s/m}) - \log G_q(e^{\xi/m}) +O\left( \frac{u^2 \log^2(u+1)}{m^2 q^{2\lceil \frac{m+1}{2}\rceil}}\right)\\
&=\sum_{j=m+1}^{2m} \frac{a_j}{j}(e^{-js/m}-e^{j\xi/m}) + O\left(\frac{u^2\log^2(u+1)}{mq^{m+1}}\right).
\end{split}
\end{equation}
Hence the integral in \eqref{eq:finalboss} equals
\[ =\sum_{j=m+1}^{2m} \frac{a_j}{j} \int_{-\xi+i(1+u\xi)}^{-\xi+i\pi m} \frac{e^{us}}{s} (e^{-js/m}-e^{j\xi/m}) {\mathrm d}s+O\left( \frac{u^2 \log^2(u+1) e^{-u\xi}\log(m+1)}{mq^{m+1}} \right).
\]
If $u < 3$ then $G_q(e^{\xi/m})-1\ll 1/(mq^{\lceil (m+1)/2\rceil})$, implying Theorem~\ref{thm:smallupnmq} is already in \eqref{eq:combined}. From now on we assume $u \ge 3$. We have the estimate \[ \int_{-\xi+i(1+u\xi)}^{-\xi+i\pi m} \frac{e^{(u-j/m)s}}{s}{\mathrm d}s \ll \frac{e^{-(u-j/m)\xi}}{(1+u\xi)|u-j/m|},\]
which is valid for all $2m\ge j \ge 0$ and $u \ge 3$, and is established by integration by parts. It follows that the integral \eqref{eq:finalboss} contributes 
\[\ll  \frac{u^2 \log^2(u+1) e^{-u\xi}\log (m+1)}{mq^{m+1}}  + \sum_{j=m+1}^{2m} \frac{a_j}{j} \frac{e^{-u\xi}}{1+u\xi}  \frac{e^{j\xi/m}}{|u-j/m|} \ll \frac{e^{-u\xi}}{n q^{\lceil \frac{m+1}{2}\rceil}}.\]
Collecting the estimates,
\begin{equation}\label{eq:collect}
X \ll \frac{\rho(u)}{u}(|b_1|+b_2)+(b_3+b_4) e^{-u\xi}+ \frac{e^{-u\xi}}{n q^{\lceil \frac{m+1}{2}\rceil}}.
\end{equation}
Recall we want to show $X \ll \rho(u) \log(u+1)/(mq^{\lceil(m+1)/2\rceil})$. We have $e^{-u\xi}\ll_k \rho(u) u^{-k}$ for any $k$ by \eqref{eq:rho and i}, showing that the last term in \eqref{eq:collect} is acceptable and that it suffices to show $b_i \ll u\log(u+1)/(mq^{\lceil(m+1)/2\rceil})$ for $i=1,2,3,4$.

Since $T(-\xi+it)$ is bounded when $|t|\le \pi m$ by Corollary \ref{cor:t}, it follows that $b_3 \ll b_4$. By \eqref{eq:H} and the triangle inequality, $b_4 \ll u\log(u+1)/(mq^{\lceil(m+1)/2\rceil})$. To bound $b_1$ and $b_2$ we use the fact that $T$ and its derivatives are bounded by Lemma \ref{lem:Tderivs} in order to reduce the problem to showing \[ H^{(i)}(-\xi+it) \ll u \log(u+1)/(mq^{\lceil(m+1)/2\rceil})\]
holds for $i=0,1,2$. For $i=0$ this is in \eqref{eq:H}, for $i=1$ we have
\[ H'(s) = -\frac{e^{-s/m}}{m} \frac{G_q'(e^{-s/m})}{G_q(e^{\xi/m})} \ll \frac{|G_q'(e^{-s/m})|}{m} \ll \frac{\sum_{i=m+1}^{\infty} a_i e^{(i-1)\xi/m}}{m} \ll \frac{u\log(u+1)}{mq^{\lceil(m+1)/2\rceil}}\]
and a similar computation holds for $i=2$. This finishes the proof of Theorem~\ref{thm:smallupnmq}.
\subsection{Proof of Corollary~\ref{cor:rangelarger}}
Fix $\varepsilon>0$. We assume $n \ge m \ge (2+\varepsilon)\log_q n$ and we want to establish \eqref{eq:main mediumn32} in this range. We already know \eqref{eq:main mediumn32} holds in $n/(\log n \log^3\log(n+1)) \ge m \ge (2+\varepsilon)\log_q n$, so we may suppose $n \ge m >n/(\log n \log^3\log(n+1))$. 

If $n=O(1)$ then \eqref{eq:main mediumn32} becomes $p_{n,m,q}=p_{n,m}G_q(x) + O_{n,\varepsilon}(G_q(x)/q^{\lceil (m+1)/2\rceil})$. To establish this, recall $p_{n,m,q}=p_{n,m} + O(1/q^{\lceil (m+1)/2\rceil})$ by \cite[Prop.~1.5]{Gorodetsky2022} and observe that $G_q(x) = 1+O_{n,\varepsilon}(1/q^{\lceil (m+1)/2\rceil})$ by \eqref{eq:Gqnotation} and \eqref{eq:aisize}.

From now on we may suppose $n$ is sufficiently large. In particular, $n>\sqrt{10 n \log n}$ holds, and $n \ge m >n/(\log n \log^3\log(n+1))$ implies $n \ge m \ge \sqrt{10 n \log n}$. We apply Theorem~\ref{thm:smallupnmq}, and see that it suffices to show that 
\begin{equation}\label{eq:GqxGqxi}
	\frac{G_q(x)}{G_q(e^{\xi(u)/m})} = 1 +O\left( \frac{u\log^2 (u+1)}{m^2 q^{\lceil \frac{m+1}{2}\rceil}}\right)
\end{equation}
holds for $n\ge m \ge \sqrt{10 n \log n}$.
First we verify that $x<\sqrt[3]{2}$. Since $x \le n^{1/m}$, it suffices to show that $2^m>n^3$, which follows from $n \ge m \ge \sqrt{10n\log n}$. Now that we know $G_q$ converges absolutely at $x$ we proceed. The relation $x^m = e^{\xi}(1+O(\log(u+1)/m))$ from \eqref{eq:xidiff} implies
\[ \log G_q(x) - \log G_q(e^{\xi/m}) = \sum_{i=m+1}^{\infty} \frac{a_i (x^i - e^{\xi i /m})}{i} \ll \frac{u\log^2 (u+1)}{m^2 q^{\lceil \frac{m+1}{2}\rceil}}\]
using \eqref{eq:aisize}, and \eqref{eq:GqxGqxi} follows by exponentiating.

\subsection{Auxiliary computation}
Recall $F_q$ defined in \eqref{eq:defFq}. We define $x_q:=x_{n,m,q}<q$ as the unique positive solution to $z F_q'(z)/F_q(z) = n$.
The next lemma gives precise results on $x-x_q$ in certain ranges.
\begin{lem}\label{lem:xxq}
If $n \ge m \ge 1$ then $x_q \le x$.
Fix $q$ and $\varepsilon>0$. For $n \ge C_{q,\varepsilon}$ and $m \in [(3/4)\log_q n, (2-\varepsilon)\log_q n]$,
\[ C_{q,\varepsilon}\frac{\min\{n,q^m\}}{q^m m} \big(1-\frac{x_q}{q}\big)^{-1} \ge  x-x_q \ge c_q \frac{\min\{n,q^m\}^{2}}{q^m nm}.\]
\end{lem}
\begin{proof}
We shall use the form of $G_q$ given in \eqref{eq:Gqnotation}. Let $f(z):=z(\log F(z))' = \sum_{i=1}^{m} z^i$ and $g_q(z) := z(\log G_q(z))'= \sum_{i=m+1}^{\infty}a_i z^i$. By definition, \begin{equation}\label{eq:nfg}
	n = f(x)= f(x_q)+g_q(x_q).
\end{equation}
Since $f$ and $g_q$ have nonnegative coefficients, the first part of the lemma follows at once. We now assume $(2-\varepsilon)\log_q n \ge m \ge (3/4)\log_q n $.
We use \eqref{eq:nfg} to determine the size of $x-x_q$. By the mean value theorem,
\[ 0 \le x-x_q = f^{-1}(n) - f^{-1}(n-g_q(x_q)) = \frac{g_q(x_q)}{f'(f^{-1}(n-t))} \]
for some $t \in [0,g_q(x_q)]$, where $f^{-1}$ is the inverse of $f$. 
Since $f^{-1}$ and $f'$ are monotone increasing we have $f'(f^{-1}(n-t)) \in [f'(x_q), f'(x)]$ and so
\[	x-x_q \in \left[ \frac{g_q(x_q)}{f'(x)}, \frac{g_q(x_q)}{f'(x_q)} \right]. \]
We have the trivial bounds $1 \le x\le n^{1/m} \ll_q 1$ and so $f'(x) = \lambda/x \asymp_q nm$ by Lemma~\ref{lem:xlam}. We also have
\[ g_q(x_q) \ge a_{2m}x_q^{2m} \gg q^{-m} x_q^{2m}\]
since $a_i \ge 0$ and $a_{2m} \gg q^{-m}$ by \cite[Lem.~2.1]{Gorodetsky2022}. It follows that
\[ x-x_q \gg_q \frac{x_q^{2m}}{q^m nm}.\]
Lemma 2 and Theorem 2 of \cite{Manstavicius1992} estimate $x_q$ and yield that, when $2\log_q n \ge m \ge (3/4)\log_q n$, 
\begin{equation}\label{eq:xqmsize}
	x_q^m \asymp_q \min\{n,q^m\},
\end{equation}
implying the desired lower bound on $x-x_q$. To prove the upper bound, first observe that $m \le (2-\varepsilon)\log_q n$ implies, via \eqref{eq:xqmsize}, that $x_q \ge q^{1/2}(1+c_{q,\varepsilon})$ if $n$ is sufficiently large, and so 
\[ g_q(x_q) \ll \sum_{i=m+1}^{2m} \frac{x_q^i}{q^{i/2}}+q^m \sum_{i \ge 2m}\frac{x_q^i}{q^i} \ll_{q,\varepsilon} \frac{ x_q^{2m}}{q^m}\big(1-\frac{x_q}{q}\big)^{-1}\]
using the bounds in \eqref{eq:aisize}. To conclude we need to lower bound $f'(x_q)$. We have
$f'(x_q) = x_q^{-1} \sum_{i=1}^{m} i x_q^i \gg_{q} m x_q^m$.
\end{proof}
\subsection{Proof of Theorem~\ref{thm:trans}}
Recall the functions $F$, $F_q$ and $G_q$ are defined in \eqref{eq:Fdef}, \eqref{eq:defFq} and \eqref{eq:Gqdeforig}. If $m \to \infty$ and $u\ge (\log \log m)^3 \log m$, Manstavi\v{c}ius proved in \cite{Manstavicius1992,Manstavicius19922} that
\begin{equation}\label{eq:manst}
	p_{n,m,q} =  \frac{F_q(x_{q})}{x_q^n\sqrt{2\pi \lambda_q}}\left( 1+O_q\left( \frac{m}{n}+\frac{m}{q^m}\right)\right)
\end{equation}
where
\[ \lambda_q:=\lambda_{n,m,q}= z \left(\frac{zF'_q(z)}{F_q(z)}\right)'|_{z=x_q}.\]
We divide \eqref{eq:manst} by \eqref{eq:saddle}, and use the fact that $\lambda$ is asymptotic to $nm$ by Lemma~\ref{lem:xlam} as long as $u \to \infty$, and 
\begin{equation}\label{eq:lambdaq}
\lambda_q \sim nm \left(1+ \frac{n}{q^m}\left(1-\frac{1}{q}\right)\right)
\end{equation}
by Theorem 2 of \cite{Manstavicius1992} as long as $u \to \infty$ and $m \to \infty$, to obtain
\begin{equation}\label{eq:ppGratio} \frac{p_{n,m,q}}{p_{n,m}G_q(x)} \sim \left(1+ \frac{n}{q^m}\left(1-\frac{1}{q}\right)\right)^{-1/2} \frac{x^nF(x_{q})G_q(x_{q})}{x_q^nF(x)G_q(x)}
\end{equation}
as $n \to \infty$ if $(2-\varepsilon)\log_q n \ge m > \log_q n$. Letting 
\[ H_q(z) :=\log F_q(z)-n\log z = \log F(z)  + \log G_q(z)-n \log z,\]
we can write \eqref{eq:ppGratio} as
\begin{equation}\label{eq:ppGrationicer} \frac{p_{n,m,q}}{p_{n,m}G_q(x)} \sim  \exp\left(H_q(x_q)-H_q(x)\right) \left(1+ \frac{n}{q^m}\left(1-\frac{1}{q}\right)\right)^{-1/2}.
\end{equation} 
By definition,  $H'_q(x_q)=0$. By Taylor-expanding $H_q$ at $x_q$,
\[ H_q(x)-H_q(x_{q}) = \frac{H''_q(t)}{2}(x-x_q)^2 \]
for some $t \in [x_q,x]$ (recall $x_q \le x$ by Lemma~\ref{lem:xxq}). In the range $(2-\varepsilon)\log_q n \ge m >\log_q n$ we have $x\asymp_q x_q \asymp_q 1$ by Lemma \ref{lem:xlam} and \eqref{eq:xqmsize}. In the notation of \eqref{eq:Gqnotation},
\[ t^2 H_q''(t) = \sum_{i=2}^{m}(i-1)t^i + \sum_{i=m+1}^{\infty}(i-1)a_i t^i+ n>0\]
is increasing for $t>0$ since $a_i \ge 0$. It follows that 
\[ x_q^2 H_q''(x_q) \ll_q H_q''(t)\ll_q x^2 H_q''(x).\]
A short computation shows that
$x_q^2 H_q''(x_q)=\lambda_q$ holds by the definition of $x_q$ and $\lambda_q$. By \eqref{eq:lambdaq} it then follows that $x_q^2 H_q''(x_q) \asymp_q nm$ when $(2-\varepsilon)\log_q n \ge m >\log_q n$. In the same range we find that $x^2 H_q''(x) =  \lambda  +\sum_{i=m+1}^{\infty}(i-1)a_i x^i \ll_{q,\varepsilon} nm$ using Lemma \ref{lem:xlam} and the bounds in \eqref{eq:aisize}. Hence $H_q''(t) \asymp_{q,\varepsilon} nm$ and
\[ H_q(x)-H_q(x_q) \asymp_{q,\varepsilon} nm (x-x_q)^2 \asymp_{q,\varepsilon} \frac{n^3}{mq^{2m}}\]
where we used Lemma \ref{lem:xxq} in the second estimate. Plugging this estimate in \eqref{eq:ppGrationicer}, and observing that $(1+ \frac{n}{q^m}(1-\frac{1}{q}))^{-1/2} \sim 1$ holds in the range $m\ge (1+\varepsilon)\log_q n$ considered in Theorem~\ref{thm:trans}, concludes the proof.
\subsection{A variant}
We prove a variant of Theorem~\ref{thm:trans} with the main term $G_q(x)$ replaced by  $G_q(x_q)$. 
\begin{thm}\label{thm:trans2}
	Fix $q$ and $\varepsilon>0$. For $m \in [(3/4)\log_q n , (2-\varepsilon)\log_q n]$ we have 
	\begin{equation}\label{eq:Gqlowerupper}
		\begin{split}
			\frac{p_{n,m,q}}{p_{n,m}} &\le (1+o(1))G_q(x_q)\exp\left( \frac{C_{q,\varepsilon} \min\{n,q^m\}^2 n}{mq^{2m} } \left(1-\frac{x_q}{q}\right)^{-2}\right),\\
			\frac{p_{n,m,q}}{p_{n,m}} &\ge (1+o(1)) G_q(x_q)\exp\left( \frac{c_{q,\varepsilon} \min\{n,q^m\}^5}{n^2 mq^{2m}}\right)
		\end{split}
	\end{equation}
	as $n \to \infty$. If furthermore $n \ge C_{q,\varepsilon}$ then
	\[C_{q,\varepsilon}\frac{\min\{n,q^m\}^2}{mq^m}\big(1-\frac{x_q}{q}\big)^{-1} \ge \log G_q(x_q) \ge c_q  \frac{\min\{n,q^m\}^2}{mq^m}.\]
\end{thm}
\begin{proof}
In very much the same way \eqref{eq:ppGratio} is proved, we also have
\begin{equation}\label{eq:ppGratio2} \frac{p_{n,m,q}}{p_{n,m}G_q(x_q)} \sim \frac{x^nF(x_{q})G_q(x_{q})}{x_q^nF(x)G_q(x_q)}\left(1+ \frac{n}{q^m}\left(1-\frac{1}{q}\right)\right)^{-1/2}
\end{equation}
if $u \ge (\log \log m)^3 \log m$ and $m\to\infty$. Letting 
\[ H(z) := \log F(z)  -n \log z,\]
we can write \eqref{eq:ppGratio2} as
\[ \frac{p_{n,m,q}}{p_{n,m}G_q(x_q)} \sim \exp\left(H(x_q)-H(x)\right)\left(1+ \frac{n}{q^m}\left(1-\frac{1}{q}\right)\right)^{-1/2}.\]
By definition, $H'(x) = 0$. By Taylor-expanding $H$ at $x$,
\[  H(x_q)-H(x) = \frac{H''(t)}{2}(x-x_q)^2\]
for some $t \in [x_q,x]$. Since $t^2 H''(t) =  \sum_{i=2}^{m} (i-1)t^i +n$ is increasing in $t$, and $x \asymp_q x_q \asymp_q 1$ by Lemma \ref{lem:xlam} and \eqref{eq:xqmsize}, it follows that
\[  x_q^2 H''(x_q) \ll_q H''(t) \ll_q x^2 H''(x).\]
We have
$x^2 H''(x)=\lambda \asymp nm$ by Lemma \ref{lem:xlam} when $(2-\varepsilon)\log_q n \ge m \ge (3/4)\log_q n$. Lemma~\ref{lem:xxq} bounds $x-x_q$. This gives the first part of \eqref{eq:Gqlowerupper} (we bound the factor $(1+ \frac{n}{q^m}(1-\frac{1}{q}))^{-1/2}$ by $1$). We have $x_q^2 H''(x_q) \ge (m-1)x_q^m \gg_q m \min\{n,q^m\}$ by \eqref{eq:xqmsize}, which leads to the second part of \eqref{eq:Gqlowerupper} (we absorb $(1+ \frac{n}{q^m}(1-\frac{1}{q}))^{-1/2}$ in the exponential factor in right-hand side of \eqref{eq:Gqlowerupper}).

It remains to estimate $G_q(x_q)$. For a lower bound we use $\log G_q(x_q) \ge a_{2m} x_q^{2m}/(2m) \gg x_q^{2m}/(mq^m)$. Here we used $a_i \ge 0$ and $a_{2m} \asymp q^{-m}$ \cite[Lem.~2.1]{Gorodetsky2022}. This is simplified using \eqref{eq:xqmsize}. For the upper bound we use $a_{i} \ll \min\{q^{-\lceil i/2\rceil},q^{m-i}\}$ and $x_q \ge q^{1/2}(1+c_{q,\varepsilon})$ to obtain $\log G_q(x_q) \ll_{q,\varepsilon} x_q^{2m}/( mq^m(1-x_q/q))$. We again simplify $x_q^m$ using \eqref{eq:xqmsize}.
\end{proof}

\subsection*{Acknowledgements}
We thank Andrew Granville for valuable comments on the presentation. We are grateful to the anonymous referees for helpful comments on exposition as well as a simplification of the proof of Lemma~\ref{lem:xlam}. This project has received funding from the European Research Council (ERC) under the European Union's Horizon 2020 research and innovation programme (grant agreement No 851318).

\bibliographystyle{abbrv}
\bibliography{references}

\Addresses
\end{document}